\documentclass[11pt]{article}
\usepackage[margin=1.1in]{geometry}
\usepackage{amsmath,amssymb,amsthm}
\usepackage{enumitem}
\usepackage{tikz}
\usetikzlibrary{arrows.meta}
\usepackage[colorlinks=true,citecolor=blue,linkcolor=black,urlcolor=black]{hyperref}
\theoremstyle{definition}
\newtheorem{theorem}{Theorem}[section]
\newtheorem{lemma}[theorem]{Lemma}
\newtheorem{proposition}[theorem]{Proposition}
\newtheorem{conjecture}[theorem]{Conjecture}
\newtheorem{corollary}[theorem]{Corollary}
\newtheorem{observation}[theorem]{Observation}
\newtheorem{definition}[theorem]{Definition}
\newtheorem{remark}[theorem]{Remark}
\title{Connectivity of the reconfiguration graph of Hamiltonian paths}
\author{Albi Kazazi\\
\small Department of Mathematics and Statistics, York University\\
\small \texttt{ak87@yorku.ca}}
\date{September 24, 2026}
\begin{document}
\maketitle

\begin{abstract}
A backbite on a Hamiltonian path of a graph $G$ adds an edge of $G$ from an endpoint $u$
of the path to a vertex $v$ and deletes an edge of the path at $v$ so that the result is
again a Hamiltonian path. We study the connectivity of the reconfiguration graph
$R(G)$, whose vertices are the Hamiltonian paths of $G$, with two adjacent when a single
backbite carries one to the other. We prove that $R(G)$ is connected whenever $G$ contains the
square of a Hamiltonian path as a spanning subgraph, and that any two Hamiltonian
paths can be reconfigured using $O(n^2)$ backbites. Combining the theorem with known
results on spanning squared paths gives connectivity for every graph with
$\delta(G)\ge\lceil\frac{2n-1}3\rceil$, for the circulants $\mathrm{Circ}(n;S)$ with
$a,2a\in S$ for some $a$ coprime to $n$, and, with high probability, for the random
graph $G(n,p)$ with $p\ge c/\sqrt n$, for a sufficiently large constant $c$. In the other direction, graphs with arbitrarily high
connectivity may have disconnected reconfiguration graphs. For every $n\ge8$ and every
$1\le m\le\lfloor n/2\rfloor-3$ we construct a graph on $n$ vertices with
$\kappa=\lambda=\delta=m$ whose reconfiguration graph is disconnected, with at least
$(n-2m-4)!$ components. We conjecture that $\delta(G)\ge n/2$ suffices for
connectivity.
\end{abstract}

\section{Introduction}
A \emph{backbite move} on a Hamiltonian path $w_1\cdots w_n$ of a graph $G$ chooses an
endpoint, say $w_1$, and an index $k\ge3$ with $\{w_1,w_k\}\in E(G)$, adds the edge
$\{w_1,w_k\}$, and deletes $\{w_{k-1},w_k\}$; the result is again a Hamiltonian path. The
\emph{reconfiguration graph} $R(G)$ has the Hamiltonian paths of $G$ as its vertices, with
two joined when a single backbite carries one to the other, so that distance in $R(G)$
counts backbites. A Hamiltonian path and its reversal are the same vertex of $R(G)$.
We study when $R(G)$ is connected, that is, when any two Hamiltonian
paths of $G$ can be reconfigured into one another by a sequence of backbites.

The backbite move comes from the Monte Carlo simulation of lattice polymers, where it
is used to sample self-avoiding walks \cite{Man}. Oberdorf et al. later used the move to
sample compact lattice walks and observed that its general ergodicity was not known \cite{OFJK}.
The same operation, applied to paths that need not be spanning, is the rotation of
P\'osa's rotation-extension method \cite{Pos}.
For a survey of reconfiguration problems see \cite{Nis} and \cite{vdH}.
A \emph{switch} of size $k$ replaces $k$ edges of a Hamiltonian cycle or path of $G$
by $k$ edges of $G$ not on it, so that the result is again a Hamiltonian cycle or path.
Takaoka \cite{Tak} proved that deciding whether two
Hamiltonian cycles of a graph are related by a sequence of switches of size two is
PSPACE-complete,
a question posed implicitly by Ito et al. \cite{Ito} and raised again in \cite{vdH}.
Kleer, Patel and Stroh \cite{KPS} proved that any two Hamiltonian cycles of a graph
with $\delta(G)\ge n/2+7$ are related by switches of size at most ten, and gave graphs
with $\delta(G)\ge 2n/3-1$ having two Hamiltonian cycles that no sequence of switches of
size two relates. For fixed-endpoint subclasses in rectangular grids, connectivity under
switches of size two was proved for simple and $1$-complex paths by Nishat, Srinivasan
and Whitesides \cite{NSWsimple,NSWcomplex}. The author subsequently proved connectivity
for all rectangular-grid Hamiltonian paths under a move set consisting of switches of
size two, double switches, and backbites \cite{Kaz}.

Bousquet et al. \cite{BIK} reconfigure spanning trees by edge flips, an operation that deletes one edge and adds another.
They show that the problem is PSPACE-complete for every bound $d\ge3$ on the maximum
degree of the trees, and they leave the case $d=2$ open. A spanning tree of maximum degree
$2$ is a Hamiltonian path, so that case is the problem studied here. A backbite is an edge flip. Conversely, the only flip between Hamiltonian paths that is not itself a backbite
joins the two endpoints, and its effect can be realized by rotating the resulting Hamiltonian
cycle. Thus edge flips and backbites give the same reachability relation for Hamiltonian paths.

Theorem~\ref{thm:main2} below is the main result. It identifies a structural hypothesis that forces connectivity of $R(G)$.
For an ordering $v_1,\dots,v_n$ of a finite set, the \emph{square} of the path
$v_1\cdots v_n$, denoted $P_n^2$, is the graph on $\{v_1,\dots,v_n\}$ whose edges join the
vertices at index distance $1$ or $2$.

\begin{theorem}[Squared-path theorem]\label{thm:main2}
Let $G$ be a graph on $n$ vertices containing the square of a Hamiltonian path as a spanning
subgraph. The reconfiguration graph $R(G)$ is connected under backbite moves, with diameter $O(n^2)$.
\end{theorem}

\noindent Containing a spanning squared path is preserved under adding edges. If $P_n^2\subseteq G$
and $G\subseteq G'$ on the same vertex set, then $P_n^2\subseteq G'$. So any sufficient
condition for containing a spanning squared path is itself a sufficient condition for
connectivity of $R(G)$, and Corollary~\ref{cor:densities} below records three instances. We first define circulant graphs.

For a set $S\subseteq\mathbb{Z}_n\setminus\{0\}$ with $S=-S$, the \emph{circulant}
$\mathrm{Circ}(n;S)$ is the graph with vertex set $\mathbb{Z}_n$ in which $i$ and $j$ are
adjacent when $i-j\in S$. We write $\mathrm{Circ}(n;j,k)$ for
$\mathrm{Circ}(n;\{\pm j,\pm k\})$, the circulant in which each vertex $i$ is adjacent to
$i\pm j$ and $i\pm k$. For a graph $F$, the \emph{square} $F^2$ is the graph on
$V(F)$ in which two distinct vertices are adjacent whenever their distance in $F$ is
at most two.

\begin{corollary}\label{cor:densities}
If $G$ has a Hamiltonian path, then $R(G^2)$ is connected. Moreover, $R(G)$ is
connected whenever any of the following hold:
\begin{enumerate}[label=(\alph*),topsep=2pt,itemsep=1pt]
\item $\delta(G)\ge\lceil\frac{2n-1}3\rceil$;
\item $G=\mathrm{Circ}(n;S)$ with $a,2a\in S$ for some $a$ coprime to $n$;
\item with high probability, $G=G(n,p)$ with $p\ge c/\sqrt n$, for a sufficiently large
constant $c$.
\end{enumerate}
\end{corollary}

\noindent In Section~\ref{sec:disconnect} we show that graphs with arbitrarily high
connectivity may have a disconnected reconfiguration graph. For every $n\ge8$ there is a graph on
$n$ vertices with minimum degree $\lfloor n/2\rfloor-3$ and disconnected
reconfiguration graph. Conjecture~\ref{conj:half} asserts that minimum degree at
least $n/2$ suffices.

\noindent The proof of Theorem~\ref{thm:main2} is constructive. For a graph $G$ satisfying the hypothesis, we define a canonical Hamiltonian path $K$ of $G$, and show that any other Hamiltonian path of $G$ can be reconfigured to $K$ by backbites. Since backbites are invertible, it follows that $R(G)$ is connected.

The reconfiguration to $K$ consists of two phases. The first phase reconfigures an arbitrary
Hamiltonian path of $G$ into one all of whose edges lie in $P_n^2$. This is
Corollary~\ref{cor:tpfer}. The second phase reconfigures a Hamiltonian path of $P_n^2$
into $K$, and rests on the fact that $P_n^2$ has exactly one Hamiltonian cycle. This
is carried out in Sections~\ref{sec:scaffold}--\ref{sec:mainproof}.

\noindent The proof of Theorem~\ref{thm:main2} occupies Sections \ref{sec:prelim2}--\ref{sec:mainproof}, and the proof of Corollary~\ref{cor:densities} is given in Section~\ref{sec:sources}. Section~\ref{sec:disconnect} constructs graphs whose reconfiguration graph is disconnected.

\section{Preliminaries}\label{sec:prelim2}

For a Hamiltonian path $H$ of $G$, we write $bb(u,v)$ for the backbite of $H$ applied at
the endpoint $u$ whose added edge is $\{u,v\}$. The order of the arguments matters
precisely when $u$ and $v$ are both endpoints of $H$, in which case it determines the
removed edge: for $H=w_1\cdots w_n$, the backbite $bb(w_1,w_n)$ removes $\{w_{n-1},w_n\}$, while the backbite $bb(w_n,w_1)$
removes $\{w_1,w_2\}$. See Figure~\ref{fig:move}.

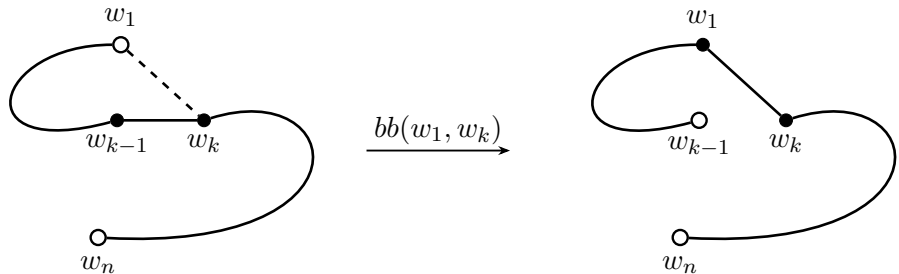
\begin{figure}[ht]
\centering
\begin{tikzpicture}[
  v/.style={circle,draw=black,fill=black,inner sep=1.7pt},
  e/.style={circle,draw=black,fill=white,line width=0.9pt,inner sep=2.0pt},
  pe/.style={black,line width=1.0pt},
  add/.style={black,line width=1.0pt,dashed}
]
\begin{scope}
\node[e,label=above:{$w_1$}]                              (a) at (1.9,2.9) {};
\node[v,label={[label distance=-1pt]below:{$w_{k-1}$}}]   (b) at (1.85,1.9) {};
\node[v,label={[label distance=-1pt]below:{$w_k$}}]       (c) at (3.0,1.9) {};
\node[e,label=below:{$w_n$}]                              (d) at (1.6,0.35) {};
\draw[pe] (a) .. controls (0.2,2.9) and (-0.2,1.35) .. (b);
\draw[pe] (b) -- (c);
\draw[pe] (c) .. controls (4.9,2.5) and (5.3,0.15) .. (d);
\draw[add] (a) -- (c);
\end{scope}
\begin{scope}[xshift=7.7cm]
\node[v,label=above:{$w_1$}]                              (a2) at (1.9,2.9) {};
\node[e,label={[label distance=-1pt]below:{$w_{k-1}$}}]   (b2) at (1.85,1.9) {};
\node[v,label={[label distance=-1pt]below:{$w_k$}}]       (c2) at (3.0,1.9) {};
\node[e,label=below:{$w_n$}]                              (d2) at (1.6,0.35) {};
\draw[pe] (a2) .. controls (0.2,2.9) and (-0.2,1.35) .. (b2);
\draw[pe] (c2) .. controls (4.9,2.5) and (5.3,0.15) .. (d2);
\draw[pe] (a2) -- (c2);
\end{scope}
\draw[-{Stealth[length=4.5pt]},line width=0.7pt] (5.15,1.5) -- (7.05,1.5)
  node[midway,above,inner sep=2pt]{$bb(w_1,w_k)$};
\end{tikzpicture}
\caption{Application of the backbite $bb(w_1,w_k)$ to a path.}
\label{fig:move}
\end{figure}

For a finite sequence or path $S=(x_1,\dots,x_k)$, $S^R$ denotes its reversal
$(x_k,\dots,x_1)$.
When one endpoint $b$ of the current path is held
fixed over a sequence of backbites, we call the opposite endpoint the \emph{free endpoint}. 

\begin{observation}[Effect of a backbite]\label{obs:blockrev}
Let $H=w_1\cdots w_n$ be a Hamiltonian path of $G$ and let $bb(w_1,w_k)$, $k\ge3$, be a
backbite of $H$. The resulting path is $w_{k-1},\dots,w_1,w_k,\dots,w_n$, and its new
endpoint is $w_{k-1}$. A backbite at $w_n$ is symmetric.
\end{observation}

\noindent We remark that, by Observation~\ref{obs:blockrev}, a backbite at one endpoint never moves the other endpoint: $bb(w_1,w_n)$ on $w_1 \cdots w_n$ gives $w_{n-1}...w_1w_n$. Furthermore a backbite is invertible, since after $bb(w_1,w_k)$, we may apply $bb(w_{k-1},w_k)$ and recover $H$.

For a Hamiltonian cycle $C$ of a graph and an edge $f=\{x,y\}$ of $C$, the path
$C-f$ has endpoints $x$ and $y$, and $bb(x,y)$ adds $f$ and removes the edge of $C$ at $y$
other than $f$, so its result is $C-f'$ for the cycle edge $f'$ neighbouring $f$ at $y$.
Successive such backbites, called \emph{rotations} of $C$, move the missing edge around $C$
in either direction.

\section{Advancing the endpoint over the squared path}\label{sec:tp}

Let $P_n=v_1\cdots v_n$ be a path on $n$ vertices, fixed for the rest of the paper, and
let $P_n^2$ be its square. Throughout this section $J$ is a simple graph on $\{v_1,\dots,v_n\}$
with $E(P_n^2)\subseteq E(J)$. For a Hamiltonian path $H$ of $J$,
$\beta(H):=|E(H)\setminus E(P_n^2)|$ counts its \emph{bad} edges. A backbite whose added
edge lies in $P_n^2$ never increases $\beta$, and decreases $\beta$ exactly when the
removed edge is bad.

Fix one endpoint of a Hamiltonian path.
Lemma~\ref{lem:force} shows that backbites at the free endpoint increase its index along the
designated ordering $v_1,\dots,v_n$, and repeating this brings the free endpoint to the
pair $\{v_{n-1},v_n\}$. Fixing that endpoint and repeating from the other end brings
both endpoints to $\{v_{n-1},v_n\}$, and the path can then be closed into a Hamiltonian
cycle by adding the edge $\{v_{n-1},v_n\}$, an edge of $P_n^2$. At each step the
alternative outcome is a drop in $\beta$. Section~\ref{sec:mainproof} applies this to $G$ and to $P_n^2$.

\begin{lemma}[Endpoint advancing, one step]\label{lem:force}
Let $n\ge3$ and let $H$ be a Hamiltonian path of $J$. Fix one endpoint $b$ of $H$, and
suppose the free endpoint of $H$ is $v_i$ with $i\le n-2$. Then there is a sequence of at most two backbites applied at the free endpoint, each with added edge in $E(P_n^2)$, that yields a path $H'$ satisfying one of:
\begin{enumerate}[label=(\alph*),topsep=2pt,itemsep=1pt]
\item $\beta(H')<\beta(H)$; or
\item the free endpoint of $H'$ is $v_{i'}$ for some $i'>i$.
\end{enumerate}
\end{lemma}

\begin{proof}
Write $H=(v_i,q_2,q_3,\dots)$ from the free endpoint. Since $i\le n-2$, both edges
$v_iv_{i+1}$ and $v_iv_{i+2}$ lie in $E(P_n^2)$, and only the second vertex $q_2$ of $H$ is
ineligible as the target of a backbite at $v_i$, so a backbite at $v_i$ with added edge in
$E(P_n^2)$ is available. The proof splits into two cases: $q_2 \neq v_{i+2}$ and $q_2 = v_{i+2}$. In each case the required backbite sequences are given explicitly, and each backbite is applied at the current free endpoint with the added edge in $E(P_n^2)$, so $\beta$ never increases. Thus if the removed edge of some applied backbite
is bad, $\beta$ decreases there, so (a) holds, and we are done. Therefore we need only consider the case where every removed edge lies in
$E(P_n^2)$.

\medskip

\noindent\textit{CASE 1. $v_{i+2}\ne q_2$.} Apply $bb(v_i,v_{i+2})$ to obtain $H'$ with
free endpoint $v'$. By the definition of $P_n^2$,
$N_{P_n^2}(v_{i+2})\subseteq\{v_i,v_{i+1},v_{i+3},v_{i+4}\}$, so
$v'\in\{v_i,v_{i+1},v_{i+3},v_{i+4}\}$. Note that by the definition of a backbite, $v'\ne v_i$, and every other index exceeds $i$,
so (b) holds after one backbite. Note that, by the remark after Observation~\ref{obs:blockrev}, the fixed endpoint $b$
remains an endpoint even when the target of the backbite is $b$ itself, as can happen
here with $b=v_{i+2}$. This holds for all backbites and will be used implicitly throughout.

\medskip

\noindent\textit{CASE 2. $v_{i+2}=q_2$.} Apply $bb(v_i,v_{i+1})$ to obtain a path $H''$ with
free endpoint $v''$. By the definition of $P_n^2$,
$N_{P_n^2}(v_{i+1})\subseteq\{v_{i-1},v_i,v_{i+2},v_{i+3}\}$, so
$v''\in\{v_{i-1},v_i,v_{i+2},v_{i+3}\}$, and $v''\ne v_i$ by the definition of a backbite. If
$v''\in\{v_{i+2},v_{i+3}\}$, then (b) holds after one backbite, with $H''=H'$.

The only remaining possibility is $v''=v_{i-1}$, which requires $i\ge2$: for $i=1$ there is
no vertex $v_{i-1}$, and the previous sentence already applies. So suppose $i\ge2$ and
$v''=v_{i-1}$. The removed edge was $\{v_{i-1},v_{i+1}\}$, so, for some possibly empty
subpaths $A$ and $C$ of $H$,
\[
H=(v_i,\;v_{i+2},\;A,\;v_{i-1},\;v_{i+1},\;C),\qquad
H''=(v_{i-1},\;A^R,\;v_{i+2},\;v_i,\;v_{i+1},\;C).
\]
Since $\{v_{i-1},v_i\}\notin E(H'')$, we may apply
$bb(v_{i-1},v_i)$, obtaining $H'=(v_{i+2},\;A,\;v_{i-1},\;v_i,\;v_{i+1},\;C)$ with free
endpoint $v_{i+2}$, so (b) holds after two backbites.

\end{proof}

\begin{corollary}[Endpoint advancing, iterated]\label{lem:tp}
Let $H$ be a Hamiltonian path of $J$ and fix one endpoint $b$ of $H$. Using backbites at the free endpoint, each with added edge in $P_n^2$, one can reach either a path with smaller $\beta$, or a path whose endpoints are $b$ and a vertex of $\{v_{n-1},v_n\}$.

\end{corollary}

\begin{proof}
Apply Lemma~\ref{lem:force} repeatedly while the free endpoint is $v_i$ with $i\le n-2$. If
outcome (a) ever occurs, we are done. Otherwise each application increases the index of the
free endpoint, which therefore reaches $\{v_{n-1},v_n\}$ after at most $n-2$ applications,
and the endpoints are then $b$ and a vertex of $\{v_{n-1},v_n\}$.
\end{proof}

\begin{corollary}[Smaller $\beta$ or adjacent endpoints]\label{cor:tpclosure}
From every Hamiltonian path of $J$, backbites with added edges in $E(P_n^2)$ reach a path with
smaller $\beta$ or a path whose endpoints are $v_{n-1}$ and $v_n$, which are adjacent in
$P_n^2$.
\end{corollary}

\begin{proof}
Fix either endpoint as $b$ and apply Corollary~\ref{lem:tp}. Unless $\beta$ drops, we
reach a path with endpoints $b$ and $w\in\{v_{n-1},v_n\}$. Now fix $w$ and apply
Corollary~\ref{lem:tp} at the opposite end. Unless $\beta$ drops, the endpoints become
$w$ and the vertex of $\{v_{n-1},v_n\}\setminus\{w\}$, so they are $v_{n-1}$ and $v_n$,
adjacent in $P_n^2$.
\end{proof}

\begin{corollary}[Reaching $P_n^2$]\label{cor:tpfer}
Backbites reconfigure every Hamiltonian path of $J$ into a Hamiltonian path of $P_n^2$.
\end{corollary}

\begin{proof}
Let $H$ be a Hamiltonian path of $J$. A path with $\beta=0$ is a Hamiltonian path of
$P_n^2$, so it suffices to show that from every path with $\beta>0$ backbites reach a
path with smaller $\beta$. Apply Corollary~\ref{cor:tpclosure} and let $H'$ be the path
reached. If $\beta(H')<\beta(H)$, we are done. Otherwise the endpoints of $H'$ are
$v_{n-1}$ and $v_n$ and $\beta(H')=\beta(H)>0$. Since
$v_{n-1}v_n\in E(P_n^2)\subseteq E(J)$, adding $v_{n-1}v_n$ to $H'$ yields a Hamiltonian
cycle $C$ of $J$, and some edge $f$ of $C$ does not lie in $E(P_n^2)$. Rotations of $C$
reconfigure $H'$ into $C-f$, whose edges are those of $H'$, plus $v_{n-1}v_n$, minus $f$,
so that $\beta(C-f)=\beta(H')-1$.
\end{proof}

\section{The canonical path}\label{sec:scaffold}
Throughout Sections~\ref{sec:scaffold}--\ref{sec:mainproof}, $G$ is a simple graph on $n$
vertices, and $v_1,\dots,v_n$ is an ordering of $V(G)$ with $E(P_n^2)\subseteq E(G)$.
Such an ordering exists by the hypothesis of Theorem~\ref{thm:main2}, and we call it a
\emph{$P_n$-ordering} of $G$. Since $E(P_n)\subseteq E(P_n^2)$, the fixed path
$P_n=v_1\cdots v_n$ is a Hamiltonian path of $G$.

\begin{definition}[Canonical path]\label{def:K2}
Let $G$ be a simple graph on $n$ vertices and let $v_1,\dots,v_n$ be a $P_n$-ordering of
$G$. Define the ordering $K$ of $V(G)$ by
\[
K=(v_1,v_3,\dots,v_{n-1},\;v_n,v_{n-2},\dots,v_4,v_2)\qquad\text{if $n$ is even},
\]
\[
K=(v_1,v_3,\dots,v_n,\;v_{n-1},v_{n-3},\dots,v_4,v_2)\qquad\text{if $n$ is odd}.
\]
\end{definition}

Figure~\ref{fig:square} illustrates the canonical path for $n=9$.

\begin{figure}[ht]
\centering
\begin{tikzpicture}[
  v/.style={circle,draw=black,fill=black,inner sep=1.5pt},
  ge/.style={black,line width=0.8pt},
  hl/.style={line width=5pt,opacity=0.30,cyan!70!blue,line cap=round,line join=round}
]
\coordinate (v1) at (0.000,0);
\coordinate (v2) at (1.200,0);
\coordinate (v3) at (2.400,0);
\coordinate (v4) at (3.600,0);
\coordinate (v5) at (4.800,0);
\coordinate (v6) at (6.000,0);
\coordinate (v7) at (7.200,0);
\coordinate (v8) at (8.400,0);
\coordinate (v9) at (9.600,0);
\draw[hl] (v1) to[bend left=60] (v3);
\draw[hl] (v3) to[bend left=60] (v5);
\draw[hl] (v5) to[bend left=60] (v7);
\draw[hl] (v7) to[bend left=60] (v9);
\draw[hl] (v9) -- (v8);
\draw[hl] (v6) to[bend right=60] (v8);
\draw[hl] (v4) to[bend right=60] (v6);
\draw[hl] (v2) to[bend right=60] (v4);
\draw[ge] (v1) -- (v2);
\draw[ge] (v2) -- (v3);
\draw[ge] (v3) -- (v4);
\draw[ge] (v4) -- (v5);
\draw[ge] (v5) -- (v6);
\draw[ge] (v6) -- (v7);
\draw[ge] (v7) -- (v8);
\draw[ge] (v8) -- (v9);
\draw[ge] (v1) to[bend left=60] (v3);
\draw[ge] (v2) to[bend right=60] (v4);
\draw[ge] (v3) to[bend left=60] (v5);
\draw[ge] (v4) to[bend right=60] (v6);
\draw[ge] (v5) to[bend left=60] (v7);
\draw[ge] (v6) to[bend right=60] (v8);
\draw[ge] (v7) to[bend left=60] (v9);
\node[v,label=below:{$v_{1}$}] at (v1) {};
\node[v,label=above:{$v_{2}$}] at (v2) {};
\node[v,label=below:{$v_{3}$}] at (v3) {};
\node[v,label=above:{$v_{4}$}] at (v4) {};
\node[v,label=below:{$v_{5}$}] at (v5) {};
\node[v,label=above:{$v_{6}$}] at (v6) {};
\node[v,label=below:{$v_{7}$}] at (v7) {};
\node[v,label=above:{$v_{8}$}] at (v8) {};
\node[v,label=below:{$v_{9}$}] at (v9) {};
\end{tikzpicture}
\caption{The squared path $P_9^2$, with the canonical path
$K=v_1v_3v_5v_7v_9v_8v_6v_4v_2$ highlighted.}
\label{fig:square}
\end{figure}
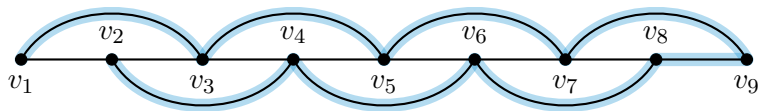

\noindent For the remainder of the paper, $K$ denotes this ordering and no other. It is the
canonical form that every Hamiltonian path of $G$ will be reconfigured into.

\begin{lemma}[The Hamiltonian cycle of $P_n^2$]\label{lem:unique}
Let $n\ge3$. Then $K$ is a Hamiltonian path of $P_n^2$, and $K+\{v_1,v_2\}$ is the only
Hamiltonian cycle of $P_n^2$.
\end{lemma}

\begin{proof}
Note that consecutive vertices of $K$ are adjacent in $P_n^2$, so $K$ is a Hamiltonian
path of $P_n^2$. Its endpoints $v_1$ and $v_2$ are adjacent in $P_n^2$, so
$K+\{v_1,v_2\}$ is a Hamiltonian cycle of $P_n^2$.

It remains to prove uniqueness. Let $C$ be a Hamiltonian cycle of $P_n^2$. We will
show that $C=K+\{v_1,v_2\}$. For $n=3$ the graph $P_3^2$ is a triangle, so this holds.
Assume $n\ge4$. The only neighbours of $v_1$ are $v_2$ and $v_3$, so $C$ contains the path
$v_2v_1v_3$. Suppose that, for some $3\le k<n$, the edges of $C$ found so far form a
path on $\{v_1,\dots,v_k\}$ with endpoints $v_{k-1}$ and $v_k$. Each of
$v_1,\dots,v_{k-2}$ already has two edges of $C$, so the second neighbour of $v_{k-1}$
on $C$ is $v_k$ or $v_{k+1}$. It is not $v_k$, since that would close a cycle on $k<n$
vertices. So $C$ contains $v_{k-1}v_{k+1}$, and the edges found so far form a path on
$\{v_1,\dots,v_{k+1}\}$ with endpoints $v_k$ and $v_{k+1}$. For $k=n-1$ this path is
spanning, with endpoints $v_{n-1}$ and $v_n$, and $C$ contains $v_{n-1}v_n$. So $C$
consists of $v_1v_2$, the edges $v_iv_{i+2}$ for $1\le i\le n-2$, and $v_{n-1}v_n$,
which are the edges of $K+\{v_1,v_2\}$.
\end{proof}

\section{Proof of Theorem~\ref{thm:main2}}\label{sec:mainproof}

For $n\le3$ the statement is immediate, so we may assume $n\ge4$. Phase one of the
proof takes an arbitrary Hamiltonian path of $G$ to a Hamiltonian path of $P_n^2$, by a
single application of Corollary~\ref{cor:tpfer}. Phase two is the following proposition.

\begin{proposition}[Reaching $K$]\label{prop:reachK}
Backbites reconfigure every Hamiltonian path of $P_n^2$ into $K$.
\end{proposition}

\begin{proof}
Let $H$ be a Hamiltonian path of $P_n^2$. Apply Corollary~\ref{cor:tpclosure} with $J=P_n^2$. Its backbites add edges of $P_n^2$, so every path they produce is a
Hamiltonian path of $P_n^2$, and no removed edge is bad. Thus the smaller-$\beta$ branch
cannot occur, and Corollary~\ref{cor:tpclosure} yields a sequence of backbites after
which the current path is a Hamiltonian path $H'$ of $P_n^2$ with endpoints $v_{n-1}$
and $v_n$. Adding the edge $\{v_{n-1},v_n\}$ to $H'$ yields a Hamiltonian
cycle of $P_n^2$, which by Lemma~\ref{lem:unique} is $C=K+\{v_1,v_2\}$. So
$H'=C-\{v_{n-1},v_n\}$. Rotations of $C$ move the missing edge to $\{v_1,v_2\}$, and
$C-\{v_1,v_2\}=K$.
\end{proof}

\medskip

\noindent\emph{Proof of Theorem~\ref{thm:main2}.} By Corollary~\ref{cor:tpfer}, applied
with $J=G$, backbites reconfigure every Hamiltonian path of $G$ into a
Hamiltonian path of $P_n^2$. By Proposition~\ref{prop:reachK}, further backbites reach
$K$. Every Hamiltonian path of $G$ thus reaches $K$, and $R(G)$ is connected.

We record the bound on the number of backbites. Reaching a Hamiltonian path of $P_n^2$
costs $O(n^2)$ backbites: each of at most $n-1$ decreases of $\beta$ uses one
application of Corollary~\ref{cor:tpclosure}, itself $O(n)$ backbites, followed by a
closing and rotation pass of $O(n)$ backbites. Proposition~\ref{prop:reachK} costs
$O(n)$ backbites: one application of Corollary~\ref{cor:tpclosure} and at most $n$
rotations. So any Hamiltonian path of $G$ reaches $K$ in $O(n^2)$ backbites, and,
backbites being invertible, any two Hamiltonian paths of $G$ are reconfigured into one
another by $O(n^2)$ backbites. \hfill$\qed$

\noindent The $O(n^2)$ bound of Theorem~\ref{thm:main2} counts backbites crudely;
determining the exact worst-case number is left open. For $P_n^2$ itself,
Proposition~\ref{prop:reachK} gives $\operatorname{diam}R(P_n^2)=O(n)$.

\section{Proof of Corollary~\ref{cor:densities}}\label{sec:sources}

For a cyclic ordering
$v_1,\dots,v_n$ of a finite set, the \emph{square} of the cycle $v_1\cdots v_nv_1$,
denoted $C_n^2$, is the graph on $\{v_1,\dots,v_n\}$ whose edges join the vertices
whose indices differ by $1$ or $2$ modulo $n$.

\medskip

\noindent\emph{Proof of Corollary~\ref{cor:densities}.}
If $v_1\cdots v_n$ is a Hamiltonian path of $F$, then $v_iv_{i+1},v_iv_{i+2}\in E(F^2)$, so
that ordering is a $P_n$-ordering of $F^2$, and Theorem~\ref{thm:main2} applies. For the
three listed conditions:
\begin{enumerate}[label=,leftmargin=1.2em,topsep=2pt,itemsep=3pt]
\item[(a)] By Fan--Kierstead \cite{FK}, every graph with
$\delta(G)\ge\lceil\frac{2n-1}3\rceil$ contains a spanning squared Hamiltonian path; apply
Theorem~\ref{thm:main2}.
\item[(b)] With $\gcd(a,n)=1$, the ordering $0,a,2a,\dots,(n-1)a$ visits all of
$\{0,\dots,n-1\}$; consecutive terms differ by $a\in S$ and terms two apart by $2a\in S$, so
this ordering spans $C_n^2\supseteq P_n^2$; apply Theorem~\ref{thm:main2}.
\item[(c)] By Kahn--Narayanan--Park \cite{KNP}, there is a constant $c>0$ such that, if
$p\ge c/\sqrt n$, then the random graph $G(n,p)$ with high probability contains the square of a
Hamilton cycle, hence a spanning copy of $P_n^2$; Theorem~\ref{thm:main2} applies to every
graph containing one.
\end{enumerate}
\hfill$\qed$

\section{Graphs with a disconnected reconfiguration graph}\label{sec:disconnect}

It is natural to ask which hypotheses weaker than that of Theorem~\ref{thm:main2}
force $R(G)$ to be connected. Edge density alone is not enough. Fix $t\ge2$ and consider
the graph $H_t$ obtained from the complete graph on the $t+2$ vertices
$a,b,h_1,\dots,h_t$ by deleting the edge $ab$ and adding two vertices $x$ and $y$,
with $x$ joined to $a$ and $y$ joined to $b$. Every Hamiltonian path of $H_t$ runs from $x$ to $y$ and admits no backbite, so
$R(H_t)$ has no edges. See Figure~\ref{fig:amp}.
Since $t\ge2$, there is more than one Hamiltonian path, and
$R(H_t)$ is disconnected. Yet $H_t$ has only $2t+4$
fewer edges than the complete graph on $t+4$ vertices, so its edge density tends to
one. The
two vertices of degree one force this, but they are not needed for disconnection. A
small grid graph with disconnected $R(G)$, the union of two $4\times4$ squares offset by
one unit, is given in \cite{Kaz}. The construction below reaches, at every order, every
connectivity and minimum degree up to $\lfloor n/2\rfloor-3$.

\begin{theorem}[Disconnection at every connectivity]\label{thm:amp}
For every $n\ge8$ and every $1\le m\le\lfloor n/2\rfloor-3$ there is a graph $G$ on
$n$ vertices with $\kappa(G)=\lambda(G)=\delta(G)=m$ whose reconfiguration graph
$R(G)$ is disconnected, with at least $(n-2m-4)!$ components.
\end{theorem}

The graph $G$ of Theorem~\ref{thm:amp} is $A_m(H_t)$, where $A_m(H)$ is defined as
follows. For a connected graph $H$ and $m\ge1$, let $S=\{s_1,\dots,s_m\}$ and
$Z=\{z_1,\dots,z_m\}$ be disjoint sets of new vertices. The graph $A_m(H)$ is the
complete bipartite graph with parts $S$ and $V(H)\cup Z$, together with the edges of
$H$. That is,
\[
V(A_m(H))=V(H)\cup S\cup Z,\qquad
E(A_m(H))=E(H)\cup\{\,uv:\ u\in S,\ v\in V(H)\cup Z\,\}.
\]
So $S$ and $Z$ are independent sets, the vertices of $Z$ have all their neighbours in
$S$, and $A_m(H)-S$ is the disjoint union of $H$ and the $m$ isolated vertices
$z_1,\dots,z_m$. The order of $A_m(H)$ is $|V(H)|+2m$. Figure~\ref{fig:amp} shows
$A_2(H_2)$.

\begin{figure}[ht]
\centering
\begin{tikzpicture}[
  v/.style={circle, draw=black, fill=black, inner sep=1.5pt},
  he/.style={black, line width=0.9pt},
  je/.style={black!35, line width=0.4pt}
]
\node[v,label=above:{$x$}]   (x)  at (0,3) {};
\node[v,label=above:{$a$}]   (a)  at (1.3,3) {};
\node[v,label=above:{$h_1$}] (h1) at (2.5,3.7) {};
\node[v,label={[label distance=-2pt]below right:{$h_2$}}] (h2) at (2.5,2.5) {};
\node[v,label=above:{$b$}]   (b)  at (3.7,3) {};
\node[v,label=above:{$y$}]   (y)  at (5.0,3) {};
\node[v,label=left:{$s_1$}]  (s1) at (1.7,1.3) {};
\node[v,label=right:{$s_2$}] (s2) at (3.3,1.3) {};
\node[v,label=below:{$z_1$}] (u1) at (1.5,0) {};
\node[v,label=below:{$z_2$}] (u2) at (3.5,0) {};
\foreach \t in {x,a,h1,h2,b,y,u1,u2}{
  \draw[je] (s1) -- (\t);
  \draw[je] (s2) -- (\t);
}
\draw[he] (x)--(a); \draw[he] (a)--(h1); \draw[he] (a)--(h2);
\draw[he] (h1)--(h2); \draw[he] (h1)--(b); \draw[he] (h2)--(b); \draw[he] (b)--(y);
\node at (-1.3,3.1) {$H_2$};
\node at (-1.3,1.3) {$S$};
\node at (-1.3,0) {$Z$};
\end{tikzpicture}
\caption{The graph $A_2(H_2)$, of order $10$. The edges of $H_2$ are black and the
edges from $S$ are grey. Deleting $S$ leaves $H_2$ and the two isolated vertices
$z_1$ and $z_2$.}
\label{fig:amp}
\end{figure}
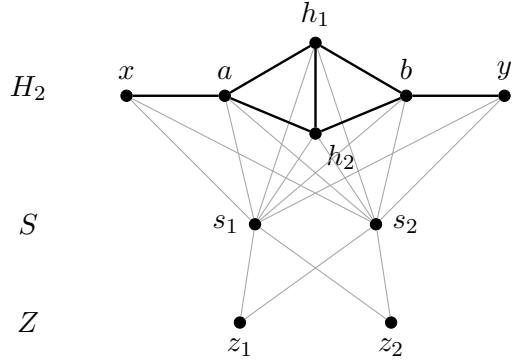

\begin{observation}[Path structure]\label{obs:ampstruct}
Let $H$ be a connected graph, let $m\ge1$, and let $P$ be a Hamiltonian path of
$A_m(H)$. Then $P-S$ has exactly $m+1$ subpaths, namely a Hamiltonian path of $H$ and
the $m$ single-vertex paths $z_1,\dots,z_m$. Moreover no vertex of $S$ is an endpoint
of $P$.
\end{observation}

\begin{proof}
Removing $m$ vertices from a path leaves at most $m+1$ subpaths. Each subpath of
$P-S$ is connected in $A_m(H)-S$ and so lies in a single component. $P$ is
Hamiltonian, so each of the $m+1$ components of $A_m(H)-S$ contains a subpath of
$P-S$. Hence $P-S$ has exactly $m+1$ subpaths, one in each component. The subpath in
the component $H$ spans $V(H)$, and its edges lie in $E(H)$, so it is a Hamiltonian
path of $H$. The subpath in $\{z_i\}$ is the single vertex $z_i$.

Suppose some vertex of $S$ is an endpoint of $P$. Deleting it does not split the
path, and deleting each of the remaining $m-1$ vertices of $S$ increases the number
of subpaths by at most one, so $P-S$ has at most $m$ subpaths, a contradiction. So no
vertex of $S$ is an endpoint of $P$.
\end{proof}

\noindent We write $P_H$ for the Hamiltonian path of $H$ contained in $P$, given by
Observation~\ref{obs:ampstruct}.

\begin{observation}[Hamiltonian paths of $H_t$]\label{obs:htpaths}
Let $t\ge2$. Then $H_t$ has exactly $t!$ Hamiltonian paths, and every one runs from
$x$ to $y$, with second vertex $a$ and second-to-last vertex $b$.
\end{observation}

\begin{lemma}[The subpath on $H_t$ is preserved]\label{lem:ampmap}
Let $t\ge2$, let $m\ge1$, and let $P$ be a Hamiltonian path of $A_m(H_t)$. Then every
backbite of $P$ leaves $E(P_{H_t})$ unchanged.
\end{lemma}

\begin{proof}
Apply a backbite at the endpoint $u$ of $P$, with target $v$. If $v$ lies in $S$, then
neither the added edge nor the removed one has both ends in $V(H_t)$, and so
$bb(u,v)$ leaves $E(P_{H_t})$ unchanged. Thus it is enough to show that $v$ lies in
$S$.
By Observation~\ref{obs:ampstruct}, $u$ is not on a vertex of $S$, so $u$ lies in $Z$
or $u$ is an endpoint of $P_{H_t}$. If $u$ lies in $Z$, then all its neighbours lie in
$S$, so $v\in S$. Otherwise, by Observation~\ref{obs:htpaths}, $u$ is on $x$ or $y$.
Suppose $u$ is on $x$. The neighbour of $x$ in $P$ is its neighbour $a$ in $P_{H_t}$,
and $N(x)=\{a\}\cup S$. It follows that $v\neq a$, and so $v$ must lie in $S$. A similar argument works if $u$ is on $y$.
\end{proof}

\begin{corollary}[Component count]\label{cor:ampmap}
Let $t\ge2$ and $m\ge1$. Then $R(A_m(H_t))$ has at least $t!$ components.
\end{corollary}

\begin{proof}
By Lemma~\ref{lem:ampmap}, $E(P_{H_t})$ remains unchanged under backbites, so it is
the same for all Hamiltonian paths in a component of $R(A_m(H_t))$. Every Hamiltonian
path $Q$ of $H_t$ extends to the Hamiltonian path
$(Q,\;s_1,\;z_1,\;s_2,\;z_2,\;\dots,\;s_m,\;z_m)$ of $A_m(H_t)$, whose subpath on
$V(H_t)$ is $Q$. Distinct paths have distinct edge sets, and $H_t$ has exactly $t!$
Hamiltonian paths by Observation~\ref{obs:htpaths}, so there are $t!$ possibilities
for $E(P_{H_t})$, and each of them occurs. Hence $R(A_m(H_t))$ has at least $t!$
components.
\end{proof}

\begin{proposition}[Connectivity and degree of $A_m(H)$]\label{prop:ampconn}
Let $H$ be a connected graph and let $m\ge1$. Then
$\kappa(A_m(H))=\lambda(A_m(H))=\delta(A_m(H))=m$.
\end{proposition}

\begin{proof}
Deleting the edges of $H$ from $A_m(H)$ leaves the complete bipartite graph with
parts $S$ and $V(H)\cup Z$, for which $\kappa=\lambda=\delta=m$. Adding edges
decreases none of the three, so $\kappa(A_m(H))\ge m$. Every vertex of $Z$ has
neighbourhood $S$, so $\delta(A_m(H))=m$. Since $\kappa\le\lambda\le\delta$, all three
equal $m$.
\end{proof}

\begin{proof}[Proof of Theorem~\ref{thm:amp}]
Let $n\ge8$ and $1\le m\le\lfloor n/2\rfloor-3$, and set $t=n-2m-4$, so that
$t\ge2$. Let $G=A_m(H_t)$, of order
$t+4+2m=n$. By Proposition~\ref{prop:ampconn}, $\kappa(G)=\lambda(G)=\delta(G)=m$.
By Corollary~\ref{cor:ampmap}, $R(G)$ has at least $t!=(n-2m-4)!$ components, and
$t!\ge2$, so $R(G)$ is disconnected.
\end{proof}

\begin{remark}\label{rem:three}
Every backbite of $A_m(H)$ either leaves $E(P_H)$ unchanged or induces a single
backbite on $P_H$, and every Hamiltonian path of $H$ extends to one of $A_m(H)$ as in
the proof of Corollary~\ref{cor:ampmap}. So for any connected graph $H$, $R(A_m(H))$
is disconnected whenever $R(H)$ is.
\end{remark}

\noindent Fixing $j\ge2$ and letting $m$ grow, the graphs $A_m(H_j)$ have
$\delta/n\to\frac12$ and at least $j!$ components.

\begin{conjecture}\label{conj:half}
Let $G$ be a graph on $n$ vertices with minimum degree at least $n/2$. Then the backbite
reconfiguration graph $R(G)$ is connected, that is, any two Hamiltonian paths of $G$ can be reconfigured into one another
by a sequence of backbite moves.
\end{conjecture}

\section*{Acknowledgements}
I thank Neal Madras for a careful reading of the manuscript and for comments that
improved the exposition, and Alan Sokal for posing the question. This research was
supported by an NSERC Discovery Grant to Neal Madras (RGPIN-2025-06291).

\end{document}